\documentclass[10pt]{amsart}
\usepackage{amsfonts,amssymb,amscd,amsmath,enumerate,verbatim,calc,amsthm}
\input xy
\xyoption{all}

\headsep=1cm \numberwithin{equation}{section}
\newtheorem{thm}{Theorem}[section]
\newtheorem{cor}[thm]{Corollary}
\newtheorem{lem}[thm]{Lemma}
\newtheorem{prop}[thm]{Proposition}
\newtheorem{defn}[thm]{Definition}

\newtheorem{rem}[thm]{Remark}

\newcommand{\Ann}{\mbox{Ann}\,}
\newcommand{\vdim}{\mbox{vdim}\,}
\newcommand{\coker}{\mbox{Coker}\,}
\newcommand{\Hom}{\mbox{Hom}\,}
\newcommand{\Ext}{\mbox{Ext}\,}
\newcommand{\Tor}{\mbox{Tor}\,}
\newcommand{\Spec}{\mbox{Spec}\,}
\newcommand{\Max}{\mbox{Max}\,}

\newcommand{\Ass}{\mbox{Ass}\,}

\newcommand{\Supp}{\mbox{Supp}\,}

\newcommand{\depth}{\mbox{depth}\,}
\renewcommand{\dim}{\mbox{dim}\,}

\newcommand{\pd}{\mbox{pd}\,}
\newcommand{\id}{\mbox{id}\,}
\newcommand{\fd}{\mbox{fd}\,}

\newcommand{\h}{\mbox{ht}\,}

\newcommand{\uhom}{{\mathbf R}\Hom}

\renewcommand{\H}{\mbox{H}}

\newcommand{\fa}{\mathfrak{a}}

\newcommand{\fm}{\mathfrak{m}}
\newcommand{\fp}{\mathfrak{p}}
\newcommand{\fq}{\mathfrak{q}}

\newcommand{\C}{C}

\begin{document}
\bibliographystyle{amsplain}


\title[Dual of Bass numbers and dualizing modules]
 {Dual of Bass numbers and dualizing modules}

\bibliographystyle{amsplain}

     \author[M. Rahmani]{Mohammad Rahmani}
     \author[A.- J. Taherizadeh]{Abdoljavad Taherizadeh}

\address{Faculty of Mathematical Sciences and Computer,
Kharazmi University, Tehran, Iran.}

\email{m.rahmani.math@gmail.com}
\email{taheri@khu.ac.ir}



\begin{abstract}
Let $R$ be a Noetherian ring and let $C$ be a semidualizing $R$-module. In this paper, we impose various conditions on $C$ to be dualizing.
For example, as a generalization of Xu \cite[Theorem 3.2]{X2}, we show that $C$ is dualizing if and only if for an $R$-module $M$, the necessary and sufficient condition for $M$ to be $C$-injective is that 
$ \pi_i(\fp , M) = 0 $ for all $ \fp \in \Spec(R) $ and all $ i \neq \h(\fp) $, where $ \pi_i $ is the invariant dual to the Bass numbers defined by E.Enochs and J.Xu \cite{EX}. 
\end{abstract}

\maketitle

\bibliographystyle{amsplain}
\section{introduction}

Throughout this paper, $R$ is a commutative Noetherian ring with non-zero identity. A finitely generated $R$-module $C$ is semidualizing if the natural homothety map $ R\longrightarrow \Hom_R(C,C) $ is an isomorphism and $ \Ext^i_R(C,C)=0 $ for all $ i>0 $. Semidualizing modules have been studied by Foxby \cite{F}, Vasconcelos \cite{V}
and Golod \cite{G} who used the name \textit{suitable}
for these modules. Dualizing complexes, introduced by A.Grothendieck, is a powerful tool for investigating cohomology
theories in algebraic geometry. A bounded complex of $R$-modules $ D $ with finitely generated homologies is said to be a dualizing complex for $R$, if  the natural homothety morphism $ R \rightarrow \uhom_R(D,D) $ is quasiisomorphism, and $ \id_R(D) < \infty $. These notion has been extended to semidualizing complexes by L.W. Christensen \cite{C}. A bounded complex of $ R $-modules $C$ with finitely generated homologies is semidualizing for $ R $ if the natural homothety morphism $ R \rightarrow \uhom_R(C,C) $ is quasiisomorphism. He used these notion to define a new homological dimension for complexes, namely G$_C  $-dimension, which is a generalization of Yassemi's G-dimension \cite{Y}. The following, is the translation of a part of  \cite[Proposition 8.4]{C} to the language of modules: \\
\textbf{Theorem 1.} Let $ (R , \fm , k) $ be a Noetherian local ring and let $ C $ be a semidualizing $R$-module. The following are equivalent:
 \begin{itemize}
 	\item[(i)]{$ C $ is dualizing.}
 	\item[(ii)]{G$_C $-$ \dim_R(M) < \infty $ for all finite $R$-modules $M$.}
 	\item[(iii)]{G$_C $-$ \dim_R(k) < \infty $.}
 \end{itemize}
In particular, the above theorem recovers \cite[1.4.9]{C2}. Note that $ k $ is a Cohen-macaulay $R$-module of type 1. R.Takahashi, in \cite[Theorem 2.3]{T}, replaced the condition G-$\dim_R(k) < \infty  $ in \cite[1.4.9]{C2} by weaker conditions and obtained a nice characterization for Gorenstein rings. Indeed, he showed that $R$ is Gorenstein, provided that either $R$ admits an ideal $I$ of finite G-dimension such that $ R/I $ is Gorenstein, or there exists a Cohen-Macaulay $R$-module of type 1 and of finite G-dimension. The following is the main result of section 3, which generalizes Theorem 1 as well as \cite[Theorem 2.3]{T}. See Theorem 3.4 below. \\
   \textbf{Theorem 2.}  Let $(R, \fm)$ be a Noetherian local ring and let $C$ be a semidualizing $R$-module. The following are equivalent:
 \begin{itemize}
 	\item[(i)]{$C$ is dualizing.}
 	\item[(ii)] {There exists an ideal $ \fa $ with G$_C$-$\dim_R(\fa C) < \infty $ such that $ C/ \fa C $ is dualizing for $ R / \fa $.}
 	\item[(iii)]{There exists a Cohen-Macaulay $R$-module $M$ with $ r_R(M) = 1 $ and G$_C$-$\dim_R(M) < \infty $.}
 	\item[(iv)]{$ r_R(C) = 1 $ and there exists a Cohen-Macaulay $R$-module $M$ with G$_C$-$\dim_R(M) < \infty $.}
 \end{itemize}
 
 E.Enochs et al. \cite{BBE}, solved a long standing conjecture about the existence of flat covers. Indeed, they showed that if $R$ is any ring, then all $R$-modules have flat covers. E.Enochs \cite{E2}, determined the structure of flat cotorsion modules. Also, E.Enochs and J.Xu \cite[Definition 1.2]{EX}, defined a new invariant $ \pi_i $, dual to the Bass numbers, for modules related to flat resolutions. J.Xu \cite{X2}, studied the minimal injective resolution of flat $R$-modules and minimal flat resolution of injective $R$-modules. He characterized Gorenstein rings in terms of vanishing of Buss numbers of flat modules, and vanishing of dual of Bass numbers of injective modules. More precisely, the following theorem is \cite[Theorems 2.1 and 3.2]{X2}. \\
   \textbf{Theorem 3.} Let $R$ be a Noetherian ring. The following are equivalent:
  \begin{itemize}
  	\item[(i)]{$ R $ is Gorenstein.}
  	\item[(ii)]{An $R$-module $F$ is flat if and only if $ \mu^i(\fp,F) = 0 $  for all $ \fp \in \Spec(R) $ whenever $ i \neq \h(\fp) $.}
  	\item[(iii)]{An $R$-module $E$ is injective if and only if $ \pi_i(\fp,E) = 0 $ for all $ \fp \in \Spec(R) $ whenever $ i \neq \h(\fp) $.}
  \end{itemize}
  In section 4, we give a generalization of Theorem 3. Indeed, in Theorem 4.3, we prove the following result. \\
     \textbf{Theorem 4.} Let $R$ be a Noetherian ring and let $C$ be a semidualizing $R$-module. The following are equivalent:
     \begin{itemize}
     	\item[(i)]{$ C $ is pointwise dualizing.}
     	\item[(ii)]{An $R$-module $M$ is $C$-injective if and only if $ \pi_i(\fp,M) = 0 $ for all $ \fp \in \Spec(R) $ whenever $ i \neq \h(\fp) $.}
     	\item[(iii)]{An $R$-module $M$ is injective if and only if $ \pi_i(\fp,\Hom_R(C,M)) = 0 $ for all $ \fp \in \Spec(R) $ whenever $ i \neq \h(\fp) $.}
      \end{itemize}
Theorem 4 has several applications. Let $ (R, \fm) $ be a $ d $-dimensional Cohen-Macaulay local ring possessing a canonical module. In this section, we give the structure of the minimal flat resolution of $ \H_{\fm}^d(R) $, the top local cohomology of $R$. More precisely, the following theorem is Corollary 4.7. \\
 \textbf{Theorem 5.}	Let $ (R, \fm) $ be a $d$-dimensional Cohen-Macaulay local ring possessing a canonical module. The minimal flat resolution of $ \H_{\fm}^d(R) $ is of the form \\
 	\centerline{$ 0 \longrightarrow \widehat{R_{\fm}} \longrightarrow \cdots \longrightarrow \underset{\h(\fp) = 1} \prod T_{\fp}  \longrightarrow \underset{\h(\fp) = 0} \prod T_{\fp} \longrightarrow \H_{\fm}^d(R) \longrightarrow 0 $,}
 	in which $ T_{\fp} $ is the completion of a free
 	$R_{\fp}$-module with respect to $ \fp R_{\fp} $-adic topology.  
 	
 In this section, by using the above resolution, we obtain the following isomorphism for a $ d $-dimensional Cohen-Macaulay local ring (See Corollary 4.8).
  	 $$\Tor^R_i(\H^d_{\fm}(R),\H^d_{\fm}(R)) \cong \left\lbrace
  	 \begin{array}{c l}
  	 \H^d_{\fm}(R)\ \ & \text{ \ \ \ \ \ $i=d$,}\\
  	 0\ \   & \text{  \ \ $\text{ \ \ $i \neq d$}$.}
  	 \end{array}
  	 \right.$$\\
\section{preliminaries}

In this section, we recall some definitions and facts which are needed throughout this
paper. By an injective cogenerator, we always mean an injective $R$-module $E$ for which $ \Hom_R(M,E) \neq 0 $ whenever $M$ is a nonzero $R$-module. For an $R$-module $M$, the injective hull of $M$, is always denoted by $E(M)$.

\begin{defn}
	\emph{Let $\mathcal{X} $ be a class of $R$-modules and $M$ an $R$-module. An $\mathcal{X}$-\textit{resolution} of $M$ is a complex of $R$-modules in $\mathcal{X} $ of the form \\
		\centerline{ $X = \ldots \longrightarrow X_n \overset{\partial_n^X} \longrightarrow X_{n-1} \longrightarrow \ldots \longrightarrow X_1 \overset{\partial_1^X}\longrightarrow X_0 \longrightarrow 0$}
		such that $\H_0(X) \cong M$ and $\H_n(X) = 0$ for all $ n \geq 1$.}
	\emph{Also the $ \mathcal{X}$-\textit{projective dimension} of $M$ is the quantity \\
		\centerline{ $ \mathcal{X}$-$\pd_R(M) := \inf \{ \sup \{ n \geq 0 | X_n \neq 0 \} \mid X$ is an $\mathcal{X}$-resolution of $M \}$}.}
	\emph{So that in particular $\mathcal{X}$-$\pd_R(0)= - \infty $. The modules of $\mathcal{X}$-projective dimension zero are precisely the non-zero modules in $\mathcal{X}$. The terms of $\mathcal{X}$-\textit{coresolution} and $\mathcal{X}$-$\id$ are defined dually.}
\end{defn}

\begin{defn}
 \emph{A finitely generated $ R $-module $ C $ is \textit{semidualizing} if it satisfies the following conditions:
\begin{itemize}
             \item[(i)]{The natural homothety map $ R\longrightarrow \Hom_R(C,C) $ is an isomorphism.}
             \item[(ii)]{$ \Ext^i_R(C,C)=0 $ for all $ i>0 $.}
          \end{itemize}}
\end{defn}
For example a finitely generated projective $R$-module of rank 1 is semidualizing. If $R$ is Cohen-Macaulay, then an $R$-module $D$ is dualizing if it is semidualizing and that $\id_R (D) < \infty $ . For example the canonical module of a Cohen-Macaulay local ring, if exists, is dualizing.

\begin{defn}
\emph{Following \cite{HJ}, let $C$ be a semidualizing $R$-module. We set
\begin{itemize}
	\item[]{$\mathcal{F}_C(R) =$ the subcategory of $R$--modules  $C \otimes_R F$ where $F$ is a flat $R$--module.}
	\item[]{ $\mathcal{I}_C(R) =$ the subcategory of $R$--modules  $\Hom_R(C,I) $ where $I$ is an injective $R$--module.}
\end{itemize}
The $R$-modules in $\mathcal{F}_C(R)$ and $\mathcal{I}_C(R)$ are called $C$-\textit{flat} and $C$-\textit{injective}, respectively.  If $ C = R $, then it recovers the classes of flat and injective modules, respectively.
We use the notations $C$-$\fd$ and $C$-$\id$ instead of $\mathcal{F}_C$-$\pd$ and $\mathcal{I}_C$-$\id$, respectively.}
\end{defn}

\begin{prop}\label{B0}
Let $C$ be a semidualizing $R$-module. Then we have the following:
\begin{itemize}
           \item[(i)]{$\emph\Supp (C) = \emph\Spec (R)$, $\emph\dim(C) = \emph\dim(R)$ and $\emph\Ass(C) = \emph\Ass(R)$.}
           \item[(ii)] { If $R \rightarrow S$ is a flat ring homomorphism, then $ C \otimes_R S$ is a semidualizing $S$-module.}
            \item[(iii)]{ If $ x \in R $ is $R$--regular, then $C/ xC$ is a semidualizing $R/ xR$-module.}
             \item[(iv)]{If, in addition, $R$ is local, then $\emph\depth_R (C) = \emph\depth (R) $.}
             \end{itemize}
\end{prop}
\begin{proof}
 The parts (i), (ii) and (iii) follow from the definition of semidualizing modules. For (iv), note that an element of $R$ is $R$-regular if and only if it is $C$-regular since $\Ass (C) = \Ass (R)$. Now an easy induction yields the equality.
\end{proof}

\begin{defn}
	\emph{Let $C$ be a semidualizing $R$-module. A finitely generated $R$-module $M$ is said to be \textit{totally} $C$-\textit{reflexive} if the following conditions are satisfied:
		\begin{itemize}
			\item[(i)]{The natural evaluation map $ M \longrightarrow \Hom_R(\Hom_R(M,C),C) $ is an isomorphism.}
			\item[(ii)]{$ \Ext^i_R(M,C) = 0 = \Ext^i_R(\Hom_R(M,C),C) $ for all $ i>0 $.}
		\end{itemize}
		For an $R$-module $ M $, if there exists an exact sequence
		$ 0 \rightarrow G_n \rightarrow \cdots \rightarrow G_1 \rightarrow G_0 \rightarrow M \rightarrow 0 $,
		of $ R $-modules such that each $ G_i $ is totally $C$-reflexive, then we say that $ M $ has G$_C $-dimension
		at most $ n $, and write G$_C $-$ \dim_R(M) \leq n$.  If there is no shorter such sequence, we set G$_C $-$ \dim_R(M) = n$. Also, if such an integer $ n $ does not exist, then
		we say that $ M $ has infinite G$_C $-dimension, and write G$_C $-$ \dim_R(M) = \infty$.}	
\end{defn}
The next proposition collects basic properties of G$ _C $-dimension. For the proof, see \cite{G}.
\begin{prop}
	Let $ (R , \fm) $ be local, $M$ a finitely generated $R$-module and let $C$ be a semidualizing $R$-module. The following statements hold:
	\begin{itemize}
		\item[(i)]{If $\emph{G}_C\emph{-\dim}_R(M) < \infty$, and $ x \in \fm $ is $M$-regular, then \\
			\centerline{$ \emph{G}_C\emph{-\dim}_R(M) = \emph{G}_C\emph{-\dim}_R(M/xM) - 1$.}
			 If, also, $ x $ is $R$-regular, then \\
			 \centerline{$ \emph{G}_C\emph{-\dim}_R(M) = \emph{G}_{C/xC}\emph{-\dim}_{R/xR}(M/xM) $.}}
		\item[(ii)] {If $\emph{G}_C\emph{-\dim}_R(M) < \infty$ and $ x $ is an $R$-regular element in $ \emph{\Ann}_R(M) $, then \\
			\centerline{$ \emph{G}_C\emph{-\dim}_R(M) = \emph{G}_{C/xC}\emph{-\dim}_{R/xR}(M) + 1 $.}}
		\item[(iii)]{Let $ 0 \rightarrow K \rightarrow L \rightarrow N \rightarrow 0 $ be a short exact sequence of $ R $-modules. If two
			of $ L,K,N $ are of finite $ \emph{G}_C $-dimension, then so is the third.}
		\item[(iv)]{If $\emph{G}_C\emph{-\dim}_R(M) < \infty$, then
			\[\begin{array}{rl}
			\emph{G}_C\emph{-\dim}_R(M) &= \sup \{ i \geq 0 \mid \emph{\Ext}^i_R(M,C) \neq 0 \} \\
			&= \emph{\depth}(R) - \emph{\depth}_R(M).\\
			\end{array}\]}
	\end{itemize}
\end{prop}
\begin{defn}
	\emph{Let $C$ be a semidualizing $R$-module. The \textit{Auslander class with
			respect to} $C$ is the class $\mathcal{A}_C(R)$ of $R$-modules $M$ such that:
		\begin{itemize}
			\item[(i)]{$\Tor_i^R(C,M) = 0 = \Ext^i_R(C, C \otimes_R M)$ for all $i \geq 1$, and}
			\item[(ii)]{The natural map $ M \rightarrow \Hom_R(C , C \otimes_R M )$ is an isomorphism.}
		\end{itemize}		
		The \textit{Bass class with
			respect to} $C$ is the class $\mathcal{B}_C(R)$ of $R$-modules $M$ such that:
		\begin{itemize}
			\item[(i)]{$\Ext^i_R(C,M) = 0 = \Tor_i^R(C, \Hom_R(C,M))$ for all $i \geq 1$, and}
			\item[(ii)]{The natural map $ C \otimes_R \Hom_R(C,M)) \rightarrow M $ is an isomorphism.}
		\end{itemize}
		The class $\mathcal{A}_C(R)$ contains all $R$-modules of finite projective dimension and those of finite $C$-injective dimension. Also the class $\mathcal{B}_C(R)$ contains all $R$-modules of finite injective dimension and those of finite $C$-projective dimension (see \cite[Corollary 2.9]{TW}). Also, if any two $ R $-modules in a short exact sequence are in $ \mathcal{A}_C(R) $ (resp. $ \mathcal{B}_C(R) $), then so is the third (see \cite{HW}).}
\end{defn}

\begin{prop}\label{A1}
Let $(R, \fm)$ be a local ring and let $C$ be a semidualizing $R$-module.
\begin{itemize}
           \item[(i)]{$C$ is a dualizing $R$-module if and only if $C \otimes_R \widehat{R}$ is a dualizing $\widehat{R}$-module }.
           \item[(ii)]{ Let $x \in \fm$  be $R$-regular. Then $C$ is a dualizing $R$-module if and only if $C/xC$ is a dualizing $R/xR$-module}.
                 \end{itemize}
\end{prop}
\begin{proof}
Just use the definition of dualizing modules.
\end{proof}

\begin{thm}\label{SQ}
Let $C$ be a semidualizing $R$-module and let $M$ be an $R$-module.

\begin{itemize}
           \item[(i)]{$C$-$\emph{\id}_R(M) = \emph{\id}_R (C \otimes_R M) $ and $\emph{\id}_R(M) =C$-$\emph{\id}_R(\emph{\Hom}_R(C,M))$}.
           \item[(ii)]{$C$-$\emph{\fd}_R(M) = \emph{\fd}_R (\emph{\Hom}_R(C,M))$ and $\emph{\fd}_R(M) =C$-$\emph{\fd}_R(C \otimes_R M) $}.
                 \end{itemize}
\end{thm}

\begin{proof}
For (i), see \cite[Theorem 2.11]{TW} and for (ii), see \cite[Proposition 5.2]{STWY}.
\end{proof}
\begin{lem}
	Let $C$ be a semidualizing $R$-module, $E$ be an injective cogenerator and $M$ be an $R$-module.
	\item[(i)]{One has $ C\emph{-\id}_R(M) = C\emph{-\fd}_R(\emph{\Hom}_R(M,E)) $.}
	\item[(ii)]{One has $ C\emph{-\fd}_R(M) = C\emph{-\id}_R(\emph{\Hom}_R(M,E)) $.}
\end{lem}
\begin{proof}
(i). We have the following equalities
 \[\begin{array}{rl}
 C\emph{-}\id_R(M) &= \id_R(C \otimes_R M) \\
 &= \fd_R(\Hom_R(C \otimes_R M , E))\\
 &=\fd_R(\Hom_R(C , \Hom_R(M,E))\\
 &=C\emph{-}\fd_R(\Hom_R(M,E)),\\
  \end{array}\]
in which the first equality is from Theorem 2.9(i), and the last one is from Theorem 2.9(ii).

(ii). Is similar to (i).
\end{proof}

\begin{rem}
	\emph{Let $ (R, \fm) $ be a local ring and let $M$ be a finitely generated $R$-module. We use $ \nu_R(M) $ to denote the minimal number of generators of $M$. More precisely, $ \nu_R(M) = \vdim_{R/ \fm}(M \otimes_R R/ \fm) $. It is easy to see that if $ x \in \fm $, then $ \nu_R(M) = \nu_{R/xR}(M/xM) $. In particular, if $ x \in \Ann_R(M) $, then $ \nu_R(M) = \nu_{R/xR}(M) $. Assume that $ \depth_R(M) = n $. The type of $M$, denoted by $ r_R(M) $, is defined to be $ \vdim_{R/ \fm}(\Ext_R^n(R/ \fm , M)) $. If $ x \in \fm $, then $  r_R(M/xM) =  r_{R/xR}(M/xM) $ by \cite[Exercise 1.2.26]{He}. Also, if $ x \in \fm $ is $M$- and $R$-regular, then $  r_R(M) =  r_{R/xR}(M/xM) $ by \cite[Lemma 3.1.16]{He}. Assuma that $C$ is a semidualizing $R$-module. Then $ r_R(C) \mid r_R(R) $. Indeed, by reduction modulo a maximal $ R $-sequence, we can assume that $ \depth_R(C) = 0 = \depth(R) $. Then we have
		 \[\begin{array}{rl}
		 r_R(R) &= \vdim_{R / \fm}\Hom_R(R / \fm ,R) \\
		 &= \vdim_{R / \fm}\Hom_R(R / \fm , \Hom_R(C ,C))\\
		 &=\vdim_{R / \fm}\Hom_R(R / \fm \otimes_R C , C)\\
		 &=\vdim_{R / \fm}\Hom_R(R / \fm \otimes_R C \otimes_{R / \fm} R / \fm , C)\\
		 &=\vdim_{R / \fm}\Hom_{R / \fm}(R / \fm \otimes_R C  , \Hom_R(R / \fm ,C))\\
		 &= \nu_R(C) r_R(C). \\
		 \end{array}\] }
\emph{In particular, if $ r_R(R) = 1 $ (e.g. $R$ is Gorenstein local), then $ \nu_R(C) = 1 $ and then $ C \cong R $.}	
\end{rem}
\begin{defn}
	\emph{Let $M$ be an $R$-module and let $ \mathcal{X} $ be a class of $R$-modules . Following \cite{EJ1}, a $ \mathcal{X} $-$precover$ of $M$ is a homomorphism
		$ \varphi : X \rightarrow M$, with $X \in \mathcal{X}$, such that every homomorphism $Y \rightarrow M$ with $Y \in \mathcal{X}$, factors through $\phi$;
		i.e., the homomorphism \\
		\centerline{$\Hom_R(Y, \varphi): \Hom_R(Y,X) \rightarrow \Hom_R(Y,M)$}
		is surjective for each module $Y$ in $ \mathcal{X} $. A $ \mathcal{X} $-precover $ \varphi : X \rightarrow M$ is a $ \mathcal{X} $-$cover$ if every $ \psi \in \Hom_R(X,X)$ with $\varphi \psi = \varphi$ is an automorphism.}
	\end{defn}
	
	\begin{defn}
		\emph{Following \cite{E2}, an $R$-module $M$ is called  \textit{cotorsion } if $ \Ext^1_R(F,M) = 0 $ for any flat $R$-module F.}
	\end{defn}
		
	\begin{rem}
	\emph{In \cite{BBE}, E. Enochs et al. showed that if $R$ is any ring, then every $R$-module has a flat cover. It is easy to see that flat cover must be surjective. By \cite[Lemma 2.2]{E2}, the kernel of a flat cover is always cotorsion. So that if $ F \rightarrow M $ is flat cover and $M$ is cotorsion, then so is $F$. Therefore for an $R$-module $M$, one can iteratively take
	flat covers to construct a flat resolution of $ M $. Since flat cover is unique up to isomorphism, this resolution is unique up to isomorphism of complexes. Such a resolution is called the minimal flat resolution of $M$. Note that the minimal flat resolution of $M$ is a direct summand of any other flat resolution of $M$. Assume that \\
	\centerline{$ \cdots \rightarrow F_i \rightarrow \cdots \rightarrow F_1 \rightarrow F_0 \rightarrow M \rightarrow 0 $,}
	is the minimal flat resolution of $M$. Then $F_i$ is cotorsion for all $ i \geq 1$. If, in addition, $M$ is cotorsion, then all the flat modules in the minimal flat resolution of $M$ are cotorsion. E. Enochs \cite{E2}, determined the structure of flat cotorsion modules. He showed that if $F$ is flat and cotorsion, then $ F \cong \underset{\fp} \prod T_{\fp} $ where $ T_{\fp} $ is the completion of a free
	$R_{\fp}$-module with respect to $ \fp R_{\fp} $-adic topology. So that we can determine the structure of the minimal flat resolution of cotorsion modules. } 	
	\end{rem}
	\begin{defn}
	\emph{Let $M$ be a cotorsion $R$-module and let \\
			\centerline{$ \cdots \rightarrow F_i \rightarrow \cdots \rightarrow F_1 \rightarrow F_0 \rightarrow M \rightarrow 0 $,}
			be the minimal flat resolution of $M$. Following \cite{EX}, for a prime ideal $ \fp $ of $R$ and an integer $ i \geq 0 $, the invariant $ \pi_i(\fp , M) $ is defined to be the cardinality
			of the basis of a free $ R_{\fp} $-module whose completion is $ T_{\fp} $ in the product
			$  F_i \cong \underset{\fp} \prod T_{\fp} $. By \cite[theorem 2.2]{EX}, for each $ i \geq 0 $, \\
			\centerline{$ \pi_i(\fp , M) = \vdim_{R_{\fp} / {\fp} R_{\fp}} \Tor^{R_{\fp}}_i \big(R_{\fp} / {\fp} R_{\fp} , \Hom_R(R_{\fp} , M) \big) $.}}
	\end{defn}
	\begin{rem}\label{A1}
		\emph{Let $M$ be a finitely generated $R$-module. There are isomorphisms\\
			\[\begin{array}{rl}
			\Hom_R(M,E(R/ \fp)) &\cong \Hom_R(M,E(R/ \fp) \otimes_R R_{\fp})\\
			&\cong \Hom_R(M,E(R/ \fp)) \otimes_R R_{\fp}\\
			&\cong \Hom_{R_{\fp}} (M_{\fp} , E_{R_{\fp}}(R_{\fp}/ \fp R_{\fp})),\\
			\end{array}\]
			where the the first isomorphism holds because $E(R/ \fp) \cong E_{R_{\fp}}(R_{\fp}/ \fp R_{\fp})$, and the second isomorphism is tensor-evaluation \cite[Theorem 3.2.14]{EJ1}.}
	\end{rem}
\section{Finiteness of G$ _C $-dimension}

Throughout this section, $C$ is a semidualizing $R$-module. We begin with three lemmas that are needed for the main result of this section. It is well-known that a local ring over which there exists a non-zero finitely generated injective module, must be Artinian. Our first lemma generalizes this fact by replacing the injectivity condition with weaker assumption.
\begin{lem}\label{A2}
Let $ (R, \fm) $ be local and let $M$ be a finitely generated $R$-module with $\emph{\depth}(M) = 0$. If $\emph{\Ext}_R^1(R/\fm , M) = 0 $, then $R$ is Artinian. In particular, $M$ is injective.
\end{lem}
\begin{proof}
We show that $ \dim(R) = 0 $. Assume, on the contrary, that $ \dim(R) > 0 $. Note that if $N$ is an $R$-module of finite length, then by using a composition series for $N$ in conjunction with the assumption, we have 
$ \Ext^1_R(N,M) = 0 $. Now an easy induction on $ \ell_R(N) $ yields the equality $ \ell_R(\Hom_R(N,M)) = \ell_R(N) \ell_R(\Hom_R(R/ \fm,M))$. Next, note that $ \ell_R(R/ \fm^i) < \infty $ for any $ i \geq 1 $, and that the sequence $ \{\ell_R(R/ \fm^i)\}^{\infty}_{i=1} $ is not bounded since $ \fm^i \neq \fm^{i+1} $ for any $ i \geq 1 $. Hence $  \{\ell_R(\Hom_R(R/ \fm^i,M))\}^{\infty}_{i=1}  $  is not bounded. But 
$ 0:_M \fm \subseteq 0:_M \fm^2 \subseteq \cdots $ is a chain of submodules of $M$, and hence is eventually stationary. This is a contradiction. Therefore $R$ is Artinian. Finally, the assumption $ \Ext^1_R(R/ \fm,M) = 0 $ implies that $M$ is injective.
\end{proof}


\begin{lem}\label{A2}
Let $ (R, \fm) $ be a local ring and let $M$ be a Cohen-Macaulay $R$-module with \emph{G}$ _C $\emph{- \dim}$_R(M) < \infty $. Then $ r_R(C) \mid r_R(M) $.	
\end{lem}
\begin{proof}
 We use induction on $ n = \depth(R) $. If $ n = 0 $, then by Proposition 2.6(iv), we have G$ _C $-$ \dim_R(M) = 0$, and hence there is an isomorphism $ M \cong \Hom_R(\Hom_R(M ,C) , C) $, and the equalities $  \depth_R(C) = 0 =  \depth_R(M)$. Hence we have
 \[\begin{array}{rl}
 r_R(M) &= \vdim_{R / \fm}\Hom_R(R / \fm ,M) \\
 &= \vdim_{R / \fm}\Hom_R(R / \fm , \Hom_R(\Hom_R(M ,C) , C))\\
 &=\vdim_{R / \fm}\Hom_R(R / \fm \otimes_R \Hom_R(M ,C) , C)\\
 &=\vdim_{R / \fm}\Hom_R(R / \fm \otimes_R \Hom_R(M ,C) \otimes_{R / \fm} R / \fm , C)\\
 &=\vdim_{R / \fm}\Hom_{R / \fm}(R / \fm \otimes_R \Hom_R(M ,C)  , \Hom_R(R / \fm ,C))\\
 &= \nu_R(\Hom_R(M,C)) r_R(C). \\
 \end{array}\]
 Therefore $ r_R(C) \mid r_R(M) $. Now, assume inductively that $ n > 0 $. We consider two cases: 
 
 \textbf{Case 1.} If $ \depth_R(M) = 0 $, then $M$ is of finite length since it is Cohen-Macaulay. Hence we can take an $R$-regular element $x$ such that $xM = 0$. Set $ \overline{(-)} = (-) \otimes_R R/xR$. Then by Proposition 2.6(ii), we have G$_{\overline{C}}$-$\dim_{\overline{R}}(M) < \infty $. Also, note that $M$ is a Cohen-Macaulay $ \overline{R} $-module. Hence by induction hypothesis we have $ r_{\overline{R}}({\overline{C}}) \mid r_{\overline{R}}(M) $. Thus $ r_R(C) \mid r_R(M) $.

 \textbf{Case 2.} If $ \depth_R(M) > 0 $, then we can take an element $ y \in \fm $ to be $M$- and $R$-regular.
 Set  $ \overline{(-)} = (-) \otimes_R R/yR$. Now $ \overline{M} $ is a Cohen-Macaulay $ \overline{R} $-module, and that \\
 \centerline{G$_{\overline{C}} $-$ \dim_{\overline{R}}(\overline{M}) =$ G$ _C $-$ \dim_R(M) < \infty$,}
 by Proposition 2.6(i). Therefore, by induction hypothesis, we have $ r_{\overline{R}}({\overline{C}}) \mid r_{\overline{R}}(\overline{M}) $, whence $ r_R(C) \mid r_R(M) $. This complete the inductive step.
\end{proof}

\begin{lem}\label{A2}
Let $ (R, \fm) $ be local and that $ r_R(C) = 1 $. If there exists a totally $ C $-reflexive $R$-module of finite length, then $ C $ is dualizing.
\end{lem}
\begin{proof}
Assume that $M$ is a finite length $ C $-reflexive $R$-module. Then $ \depth_R(M) = 0 $, and hence $ \depth(R) = $ G$_C $-$ \dim(M) = 0 $ by Proposition 2.6(iv). Therefore, we have $ \depth_R(C) = 0 $ by Proposition 2.4(iv). Now assume, on the contrary, that $ C $ is not dualizing. Hence, by Lemma 3.1, we have $ \Ext_R^1(R/ \fm , C) \neq 0 $. Let \\
\centerline{$ 0 = M_0 \subset M_1 \subset \cdots \subset M_r = M $,}
be a composition series for $M$. Thus the factors are all isomorphic to $R / \fm$, and we have exact sequences \\
\centerline{$ 0 \rightarrow M_{i-1} \rightarrow M_i \rightarrow R/ \fm \rightarrow 0 $,}
for all $ 1 \leq i \leq r $. Applying the functor $\Hom_R(-,C)$, we get the exact sequence \\
\centerline{$ 0 \rightarrow \Hom_R(R/ \fm , C)  \rightarrow \Hom_R(M_i,C) \rightarrow \Hom_R(M_{i-1} ,C) $,}
for each $ 1 \leq i \leq r-1 $. Now since $ \depth_R(C) = 0 $ and $ r_R(C) = 1 $, we have $ \Hom_R(R/ \fm , C) \cong R/ \fm $. Hence we have the inequality $ \ell_R(\Hom_R(M_i,C)) \leq \ell_R(\Hom_R(M_{i-1} ,C)) + 1 $ for each $ 1 \leq i \leq r-1 $. On the other hand, application of the  functor $\Hom_R(-,C)$ on the exact sequence $ 0 \rightarrow M_{r-1} \rightarrow M \rightarrow R/ \fm \rightarrow 0 $, yields an exact sequence \\
\centerline{$ 0 \rightarrow \Hom_R(R/ \fm , C)  \rightarrow \Hom_R(M ,C) \rightarrow \Hom_R(M_{r-1} ,C) $}
\centerline{$ \rightarrow \Ext^1_R(R/ \fm , C) \rightarrow \Ext^1_R(M , C) = 0 $.}
Therefore  $ \ell_R(\Hom_R(M,C)) = \ell_R(\Hom_R(M_{r-1} ,C)) + 1 - \ell_R(\Ext^1_R(R/ \fm , C)) $. But since $ \ell_R(\Ext^1_R(R/ \fm , C)) > 0 $, we have 
\[\begin{array}{rl}
\ell_R(\Hom_R(M,C)) &< \ell_R(\Hom_R(M_{r-1} ,C)) + 1\\
&\leq \ell_R(\Hom_R(M_{r-2} ,C)) + 2 \\
&\leq \cdots\\
&\leq \ell_R(\Hom_R(M_0 ,C)) + r \\
&= r \\
&= \ell_R(M). \\
\end{array}\]
Now since $ \Hom_R(M,C) $ is again a totally $C$-reflexive $R$-module of finite length, the same argument shows that $ \ell_R \big(\Hom_R(\Hom_R(M ,C) , C) \big) \leq \ell_R(\Hom_R(M ,C)) $. But since $M$ is totally $C$-reflexive, we have $ M \cong \Hom_R(\Hom_R(M ,C) , C) $, which implies that $ \ell_R(M) < \ell_R(M) $, a contradiction. Hence $C$ is dualizing.
\end{proof}

The following theorem is a generalization of \cite[Theorem 2.3]{T}.

\begin{thm}\label{SQ}
Let $(R, \fm)$ be local. The following are equivalent:
\begin{itemize}
           \item[(i)]{$C$ is dualizing.}
           \item[(ii)] {There exists an ideal $ \fa $ with $ \emph{G}_C\emph{-dim}_R(\fa C) < \infty $ such that $ C/ \fa C $ is dualizing for $ R / \fa $.}
            \item[(iii)]{There exists a Cohen-Macaulay $R$-module $M$ with $ r_R(M) = 1 $ and $ \emph{G}_C\emph{-dim}_R(M) < \infty $.}
             \item[(iv)]{$ r_R(C) = 1 $ and there exists a  Cohen-Macaulay $R$-module $M$ of finite $ \emph{G}_C\emph{-}$dimension.}
                 \end{itemize}
\end{thm}

\begin{proof}
   (i)$\Longrightarrow$(ii). Choose $ \fa = 0 $.

(ii)$\Longrightarrow$(iii). We show that $ C/ \fa C $ has the desired properties. First, the exact sequence \\
\centerline{$ 0 \rightarrow \fa C \rightarrow C \rightarrow  C/ \fa C \rightarrow 0 $,}
in conjunction with Proposition 2.6(iii), show that G$_C$-$ \dim_R(C/ \fa C) < \infty $. On the other hand, $ C/ \fa C $ is a Cohen-Macaulay $ R/ \fa R $-module and hence is a Cohen-Macaulay $ R$-module. Finally, by \cite[Exercise 1.2.26]{He}, we have $ r_R(C/ \fa C) = r_{R/ \fa}(C/ \fa C) = 1$.

 (iii)$\Longrightarrow$(iv). By Lemma 3.2, we have  $ r_R(C) = 1 $.

 (iv)$\Longrightarrow$(i). Assume that $ M $ is a Cohen-Macaulay $R$-module with G$ _C $-$ \dim_R(M) < \infty $. We use induction on
 $ m = \depth_R(M) $. If $ m = 0 $, then $M$ is of finite length since it is Cohen-Macaulay. Since $ \sqrt{\Ann_R(M)} = \fm $, we can choose a maximal $R$-sequence from elements of $ \Ann_R(M) $, say \textbf{x}. In view of Proposition 2.8(ii) and  Proposition 2.6(ii), we can replace $C$ by $ C/$\textbf{x}$C$ and $R$ by $ R/$\textbf{x}$R$, and assume that $ M $ is totally $ C $-reflexive. In this case, $C$ is dualizing by Lemma 3.3. Now assume inductively that $ m > 0 $. Hence $ \depth(R) > 0 $ by Proposition 2.6(iv), and we can take an element 
 $x \in\fm$ to be $M$- and $R$-regular.
 Set  $ \overline{(-)} = (-) \otimes_R R/xR$. Now $ \overline{M} $ is a Cohen-Macaulay $ \overline{R} $-module and $ r_{\overline{R}}(\overline{C}) = r_R(C) = 1 $. Also, by Proposition 2.6(i), we have G$_{\overline{C}}$-$ \dim_{\overline{R}}(\overline{M}) = $G$ _C $-$ \dim_R(M) < \infty $. Hence, by induction hypothesis, $ \overline{C} $ is dualizing for $ \overline{R} $, whence $ C $ is dualizing for $R$ by Proposition 2.8(ii).
\end{proof}

It is well-known that the existence of a finitely generated (resp. Cohen-Macaulay) module of finite injective (resp. projective) dimension implies Cohen-Macaulyness of the ring. But, in the special case that $C$ is dualizing, the proof is easy, as the following relations show \\
\centerline{$ \dim(R) = \dim_R(C) \leq \id_R(C) = \depth(R) $,}
where the first equality is from Proposition 2.4(i), and the remaining parts are from \cite[Theorem 3.1.17]{He}. Therefore, in view of Theorem 3.4, we can state the following corollary.
\begin{cor}
Let $ (R, \fm) $ be local. If there exists a Cohen-Macaulay $ R $-module of type \emph{1} and of finite \emph{G}$_C $\emph{-}dimension, then $ R $ is Cohen-Macaulay.
\end{cor}	
\section{C-injective modules}
In this section, our aim is to extend two nice results of J.Xu \cite{X2}. It is well-known that a Noetherian ring $ R $ is Gorenstein if and
 only if $ \mu^i(\fp ,R) = \delta_{i,ht(\fp)} $ (the Kronecker $ \delta $). As a generalization, J.Xu \cite[Theorem 2.1]{X2}, showed that $R$ is Gorenstein if and only if for any $R$-module $F$, the necessary and sufficient condition for $F$ to be flat is that 
 $ \mu^i(\fp , F) = 0 $ for all $ \fp \in \Spec(R) $ and all $ i \neq \h(\fp) $. Next, in \cite[Theorem 3.2]{X2}, he proved a dual for this theorem. Indeed, he proved that $R$ is Gorenstein if and only if for any $R$-module $E$, the necessary and sufficient condition for $E$ to be injective is that 
 $ \pi_i(\fp , E) = 0 $ for all $ \fp \in \Spec(R) $ and all $ i \neq \h(\fp) $. In the present section, first we generalize the mentioned results. Next, we use our new results to determine the minimal flat resolution of some top local cohomology of a Cohen-Macaulay local rings and their torsion products.
\begin{lem}
	The followings are equivalent:
	\begin{itemize}
		\item[(i)]{$C$ is pointwise dualizing.}
		\item[(ii)] { $C$-$\emph\fd_R (E(R/ \fm)) = \emph\h(\fm) $ for any $\fm \in \emph\Max(R)$. }
		\item[(iii)]{ $C$-$\emph\fd_R (E(R/ \fm)) < \infty $ for any $\fm \in \emph\Max(R)$.}
		\item[(iv)] { $C$-$\emph\fd_R (E(R/ \fp)) = \emph\h(\fp) $ for any $\fp \in \emph\Spec(R)$. }
		\item[(v)]{ $C$-$\emph\fd_R (E(R/ \fp)) < \infty $ for any $\fp \in \emph\Spec(R)$.}
		\item[(vi)] { $C$-$\emph\id_R (T_{\fm}) = \emph\h(\fm) $ for any $\fm \in \emph\Max(R)$. }
		\item[(vii)]{ $C$-$\emph\id_R (T_{\fm}) < \infty $ for any $\fm \in \emph\Max(R)$.}
		\item[(viii)] { $C$-$\emph\id_R (T_{\fp}) = \emph\h(\fp) $ for any $\fp \in \emph\Spec(R)$. }
		\item[(ix)]{ $C$-$\emph\id_R (T_{\fp}) < \infty $ for any $\fp \in \emph\Spec(R)$.}
	\end{itemize}
\end{lem}
\begin{proof}
	(i)$\Longrightarrow$(ii). Assume that $\fm \in \Max(R)$. There are equalities
	\[\begin{array}{rl}
	\C\emph{-}\fd_R (E(R/ \fm)) &= \fd_R(\Hom_R(C,E(R/ \fm)))\\
	&= \fd_{R_{\fm}}(\Hom_{R_{\fm}}(C_{\fm}, E_{R_{\fm}}(R_{\fm}/ \fm R_{\fm}))\\
	&= \id_{R_{\fm}}(C_{\fm})\\
	&= \dim(R_{\fm})\\
	&= \h(\fm),\\
	\end{array}\]
	in which the first equality is from Theorem 2.9(ii), and the second one is from Remark 2.16.
	
	(ii)$\Longrightarrow$(iii). Is clear.
	
	(iii)$\Longrightarrow$(i). We can assume that $(R, \fm)$ is local. Now one can use Theorem 2.9(ii), to see that
	\[\begin{array}{rl}
	\id_R(C) &= \fd_R(\Hom_R(C,E(R/ \fm)))\\
	&= \emph{C-}\fd_R (E(R/ \fm)) < \infty, \\
	\end{array}\]
	whence $C$ is dualizing.
	
	(i)$\Longrightarrow$(iv). Let $\fp$ be a prime ideal of $R$. Note that $E(R/ \fp)_{\fq} \neq 0$ if and only if $ \fq \subseteq \fp$. Now as in  (i)$\Longrightarrow$(ii), we have $C$-$\fd_R (E(R/ \fp)) = \dim (R_{\fp}) = \h(\fp)$.
	
	(iv)$\Longrightarrow$(v). Is clear.
	
	(v)$\Longrightarrow$(i). Again, we can assume that $R$ is local. Now the proof is similar to that of (iii)$\Longrightarrow$(i).
	
	(ii)$\Longleftrightarrow$(vi) and (iii)$\Longleftrightarrow$(vii). Note that $ T_{\fm} = \Hom_R \big( E(R/ \fm) , E(R/ \fm)^{(X)} \big) $ for some set $ X $. Now we have the equalities
	\[\begin{array}{rl}
	C\emph{-}\id_R(T_{\fm}) &= \id_R(C \otimes_R T_{\fm})\\
	&=  \id_R \big( C \otimes_R \Hom_R \big( E(R/ \fm) , E(R/ \fm)^{(X)} \big) \big) \\
	&=  \id_R \big(\Hom_R \big(\Hom_R(C , E(R/ \fm)) , E(R/ \fm)^{(X)} \big) \big) \\
	&=  \fd_R (\Hom_R(C , E(R/ \fm))) \\
	&=  C\emph{-}\fd_R (E(R/ \fm)), \\
	\end{array}\]
	in which the first equality is from Theorem 2.9(i), the fourth equality is from Remark 2.16 and the fact that $ E(R/ \fm)^{(X)} $ is an injective cogenerator in the category of $ R_{\fm} $-modules, and the last one is from Theorem 2.9(ii).
	
	(iv)$\Longleftrightarrow$(viii) and (v)$\Longleftrightarrow$(ix). Are similar to (ii)$\Longleftrightarrow$(vi).
\end{proof}
The following theorem is a generalization of \cite[theorem 2.1]{X}.
\begin{thm}
	The following are equivalent:
	
	\begin{itemize}
		\item[(i)]{$C$ is pointwise dualizing.}
		\item[(ii)] {An $R$-module $M$ is $C$-flat if and only if $\mu^i(\fp , M) = 0$ for all $\fp \in \emph{\Spec}(R)$ whenever $ i \neq \emph{\h} (\fp)$.}
		\item[(iii)] {An $R$-module $M$ is flat if and only if $\mu^i(\fp , C \otimes_R M) = 0$ for all $\fp \in \emph{\Spec}(R)$ whenever $ i \neq \emph{\h} (\fp)$.}
		\end{itemize}
\end{thm}

\begin{proof}
	(i)$\Longrightarrow$(ii). First assume that $M$ is $C$-flat. Set $ M= C \otimes_R F$, where $F$ is a flat $R$-module. Since $C$ is pointwise dualizing, we have $\mu^i(\fp , C) = 0$ for all $\fp \in \Spec(R)$ with $ i \neq \h (\fp)$. Assume that \\
	\centerline{ $ 0 \rightarrow C \rightarrow E^0(C) \rightarrow E^1(C) \rightarrow ... \rightarrow E^i(C) \rightarrow ...$}
	is the minimal injective resolution of $C$. By applying the exact functor $ - \otimes_R F$ to this resolution, we find an exact complex \\
	\centerline{ $ 0 \rightarrow M = C \otimes_R F \rightarrow E^0(C) \otimes_R F \rightarrow E^1(C) \otimes_R F \rightarrow ... \rightarrow E^i(C) \otimes_R F \rightarrow ...$, $ (*) $}
	which is an injective resolution for $M$. By \cite[Theorem 3.3.12]{EJ1} the injective $R$-module $ E(R / \fp) \otimes_R F$ is a direct sum of copies of
	$ E(R / \fp)$ for each $ \fp \in \Spec(R)$. Now, since the minimal injective resolution of $M$ is a direct summand of the complex $ (*) $, we get the result. Conversely, suppose that $M$ is an $R$-module such that $\mu^i(\fp , M) = 0$ for all $\fp \in \Spec(R)$ whenever $ i \neq \h(\fp)$. In order to show that $M$ is $C$-flat, it is enough to prove that $M_{\fm}$ is $C_{\fm}$-flat $R_{\fm}$-module for all $\fm \in \Max(R)$. For if $M_{\fm}$ is $C_{\fm}$-flat $R_{\fm}$-module for all $\fm \in \Max(R)$, then $ \Hom_R(C,M)_{\fm} \cong \Hom_{R_{\fm}}(C_{\fm},M_{\fm})$ is flat as an $R_{\fm}$-module for all $\fm \in \Max(R)$ by Theorem 2.9(ii). Hence $\Hom_R(C,M)$ is a flat $R$-module and thus $M$ is $C$-flat by Theorem 2.9(ii). Hence, replacing $R$ by $R_{\fm}$, we can assume that $(R, \fm)$ is local. Clearly we may assume that $M \neq 0$. In this case we have
	$ \id_R(M) < \infty$ since by assumption $\mu^i(\fp , M) = 0$ for all $\fp \in \Spec(R)$ and all $ i > \dim(R)$ . Hence the assumption in conjunction with Lemma 4.1, imply that $M$ has a bounded injective resolution all of whose terms have finite $C$-flat dimensions. More precisely, by Lemma 4.1, if $E^i$ is the $i$-th term in the minimal injective resolution of $M$, then $C$-fd$_R(E^i) = i$ for all $ 0 \leq i \leq \id(M)$. Breaking up this resolution to short exact sequences and using
	\cite[Corollary 5.7]{STWY}, we can conclude that $C$-fd$_R(M) = 0$. Hence $M$ is $C$-flat, as wanted.
	
	(ii)$\Longrightarrow$(iii). Assume that $M$ is a flat $R$-module. Then $ C \otimes_R M \in \mathcal{F}_C$ and
	  $\mu^i(\fp , C \otimes_R M) = 0$ for all $\fp \in \Spec(R)$ whenever $ i \neq \h (\fp)$ by assumption. Conversely, suppose that $\mu^i(\fp , C \otimes_R M) = 0$ for all $\fp \in \Spec(R)$ whenever $ i \neq \h (\fp)$. Then, by assumption, 
	  $ C \otimes_R M $ is $ C $-flat. Set $ C \otimes_R M = C \otimes_R F $, where $F$ is flat. Therefore $ C \otimes_R M \in \mathcal{B}_C(R) $, whence $ M \in \mathcal{A}_C(R) $ by \cite[Theorem 2.8(b)]{TW}. Thus we have the isomorphisms
	    \[\begin{array}{rl}
	    M &\cong \Hom_R(C,C \otimes_R M) \\
	    &\cong \Hom_R(C,C \otimes_R F)\\
	    &\cong F,\\
	    \end{array}\]
where the first and the last isomorphism hold since both $M$ and $F$ are in $ \mathcal{A}_C(R) $.

	(iii)$\Longrightarrow$(i). Note that $R$ is a flat $R$-module. Hence by assumption, if $\fm \in \Max(R)$, then $\mu^i(\fm , C \otimes_R R) = 0$ for all $ i > \h (\fm)$. Thus $ \id_{R_{\fm}}(C_{\fm}) < \infty$, as wanted.
\end{proof}

\begin{thm}
The following are equivalent:
\begin{itemize}
	\item[(i)]{$C$ is pointwise dualizing.}
	\item[(ii)]{An $R$-module $M$ is $C$-injective if and only if $ \pi_i(\fp,M) = 0 $ for all $ \fp \in \emph{\Spec}(R) $ whenever $ i \neq \h(\fp) $.}
	\item[(iii)]{An $R$-module $M$ is injective if and only if $ \pi_i(\fp,\emph{\Hom}_R(C,M)) = 0 $ for all $ \fp \in \emph{\Spec}(R) $ whenever $ i \neq \h(\fp) $.}
\end{itemize}
\end{thm}
\begin{proof}
	 (i)$ \Longrightarrow $(ii) . Assume that $M$ is a nonzero $C$-injective $R$-module. Set $ M = \Hom_R(C , E) $ with $E$ is injective. First, we show that $M$ is cotorsion. Assume that $F$ is a flat $R$-module. Then, by \cite[Theorem 3.2.1]{EJ1}, we have
	 $ \Ext_R^1(F,\Hom_R(C,E)) \cong \Hom_R(\Tor^R_1(F,C) , E) = 0 $, and hence $M$ is cotorsion.
	  Fix a prime ideal $ \fp $ of $R$ and set $ k(\fp) = R_{\fp} / {\fp} R_{\fp} $. Note that $ \Hom_R(R_{\fp},E) $ is an injective $R$-module and that $ \Hom_R(R_{\fp},E) \cong \underset{\fq \in X}\oplus E(R/ \fq) $, where $ X \subseteq \Ass_R(E) $ and each element of $X$ is a subset of $\fp$. There are isomorphisms
	      \[\begin{array}{rl}
	      \Tor^{R_{\fp}}_i \big(k(\fp) , \Hom_R(R_{\fp} , \Hom_R(C , E)) \big) &\cong \Tor^{R_{\fp}}_i \big(k(\fp) , \Hom_R(C_{\fp} , E) \big) \\
	      &\cong \Tor^{R_{\fp}}_i \big(k(\fp) , \Hom_{R_{\fp}}(C_{\fp} , \Hom_R(R_{\fp} ,E) \big)\\
	      &\cong \Tor^{R_{\fp}}_i \big(k(\fp) , \Hom_{R_{\fp}}(C_{\fp} , \underset{\fq \in X}\oplus E(R/ \fq) \big)\\
	      &\cong \Hom_{R_{\fp}}\big( \Ext^i_{R_{\fp}}(k(\fp) , C_{\fp}) ,  \underset{\fq \in X}\oplus E(R/ \fq) \big), \\
	      \end{array}\]
where the last isomorphism is from \cite[Theorem 3.2.13]{EJ1}. Now since $ C_{\fp} $ is dualizing for $ R_{\fp} $, we have $ \Ext^i_{R_{\fp}}(k(\fp) , C_{\fp}) = 0 $ for all  $ i \neq \h(\fp) $. Therefore $ \pi_i(\fp,M) = 0 $ for all $ i \neq \h(\fp) $. Conversely, assume that $M$ is a non-zero $R$-module with $ \pi_i(\fp,M) = 0 $ for all $ i \neq \h(\fp) $. By assumption, the minimal flat resolution of $M$ is of the form \\
 \centerline{$ \cdots \longrightarrow F_i \longrightarrow \cdots \longrightarrow F_1 \longrightarrow F_0 \longrightarrow M \longrightarrow 0 $,}
 in which $ F_i = \underset{\h(\fp) = i} \prod T_{\fp} $ for all $ i \geq 1 $. Also, in view of \cite[Lemma 3.1]{X2}, we have $ F_0 = \underset{\h(\fp) = 0} \prod T_{\fp} $. Hence the minimal flat resolution of $M$ is of the form \\
 \centerline{$ \cdots \longrightarrow \underset{\h(\fp) = i} \prod T_{\fp} \longrightarrow \cdots \longrightarrow \underset{\h(\fp) = 1} \prod T_{\fp}  \longrightarrow \underset{\h(\fp) = 0} \prod T_{\fp} \longrightarrow M \longrightarrow 0 $. $ (*) $} 
Let $E$ be an injective cogenerator. According to Lemma 2.10(i), it is enough to show that $\Hom_R(M,E)$ is $C$-flat. In fact, by Theorem 4.2, we need only to show that $ \mu^i(\fp , \Hom_R(M,E)) = 0$ for all $ i \neq \h(\fp) $ and all $ i \geq 0 $. Applying the exact functor
$ \Hom_R(- , E) $ on $ (*) $, we get an injective resolution \\
\centerline{$ 0 \longrightarrow \Hom_R(M,E) \longrightarrow \Hom_R\Big( \underset{\h(\fp) = 0} \prod T_{\fp} , E \Big) \longrightarrow $}
\centerline{$\Hom_R\Big( \underset{\h(\fp) = 1} \prod T_{\fp} , E \Big)  \longrightarrow  \cdots \longrightarrow \Hom_R\Big( \underset{\h(\fp) = i} \prod T_{\fp} , E \Big) \longrightarrow \cdots $,}
for $ \Hom_R(M,E) $. Note that $ \Hom_R\Big( \underset{\h(\fp) = i} \prod T_{\fp} , E \Big)  $ is an injective $R$-module for all $ i \geq 0 $. Set
$ \Hom_R\Big( \underset{\h(\fp) = i} \prod T_{\fp} , E \Big) \cong \oplus E(R/ \fq) $. We show that $ \h(\fq) = i$. Since $C$ is pointwise dualizing, by Lemma 4.1, we have $ C $-$ \fd_R(E(R/ \fq)) = \h(\fq) $. On the other hand, we have the equalities
\[\begin{array}{rl}
C\emph{-}\fd_R(E(R/ \fq)) &= C\emph{-}\fd_R( \oplus E(R/ \fq)) \\
&= C\emph{-}\fd_R\Big(\Hom_R\Big( \underset{\h(\fp) = i} \prod T_{\fp} , E \Big)\Big) \\
&= C\emph{-}\id_R\Big( \underset{\h(\fp) = i} \prod T_{\fp} \Big)\\
&= i, \\
\end{array}\]
in which the third equality is from Lemma 2.10(i), and the last one is from Lemma 4.1. Hence $ \mu^i(\fp , \Hom_R(M,E)) = 0$ for all $ i \geq 0 $ with $ i \neq \h(\fp) $, as wanted.

	(ii)$\Longrightarrow$(iii). Assume that $M$ is an injective $R$-module. Then $ \Hom_R(C,M) \in \mathcal{I}_C$ and
	$\mu^i(\fp , \Hom_R(C,M)) = 0$ for all $\fp \in \Spec(R)$ whenever $ i \neq \h (\fp)$ by assumption. Conversely, suppose
	 that $\mu^i(\fp , \Hom_R(C,M)) = 0$ for all $\fp \in \Spec(R)$ whenever $ i \neq \h (\fp)$. Then, by assumption, 
	$ \Hom_R(C,M) $ is $ C $-injective. Set $ \Hom_R(C,M) = \Hom_R(C,I) $, where $I$ is injective. Therefore $ \Hom_R(C,M) \in \mathcal{A}_C(R) $, whence $ M \in \mathcal{B}_C(R) $ by \cite[Theorem 2.8(a)]{TW}. Thus we have the isomorphisms
	\[\begin{array}{rl}
	M &\cong C \otimes_R \Hom_R(C, M) \\
	&\cong C \otimes_R \Hom_R(C, I)\\
	&\cong I,\\
	\end{array}\]
	where the first and the last isomorphism hold since both $M$ and $I$ are in $ \mathcal{B}_C(R) $. 

 (iii)$ \Longrightarrow $(i). Assume that $\fm$ is a maximal ideal of $R$. Set  $ k(\fm) = R_{\fm} / {\fm} R_{\fm} $. Since 
 $ E(R/\fm) $ is injective, by assumption, we have  $ \pi_i \big(\fm, \Hom_R(C,E(R/ \fm))\big) = 0$  fo all $ i \neq \h(\fm) $. On the other hand, there are isomorphisms
 \[\begin{array}{rl}
 \Hom_{R_{\fm}}\big( \Ext^i_{R_{\fm}}(k(\fm) , C_{\fm}) ,  E(k(\fm)) \big) &\cong \Tor^{R_{\fm}}_i \big(k(\fm) ,\Hom_{R_{\fm}}(C_{\fm} , E(k(\fm)) \big) \\
 &\cong \Tor^{R_{\fm}}_i \big(k(\fm) ,\Hom_{R_{\fm}}(C_{\fm} \otimes_{R_{\fm}} R_{\fm} , E(k(\fm)) \big) \\
 &\cong \Tor^{R_{\fm}}_i \big(k(\fm) ,\Hom_R(R_{\fm} , \Hom_{R_{\fm}}(C_{\fm} , E(k(\fm)) \big)\\
 &\cong \Tor^{R_{\fm}}_i \big(k(\fm) ,\Hom_R(R_{\fm}, \Hom_R(C , E(R/ \fm)) \big),\\
  \end{array}\]
  where the first isomorphism is from \cite[Theorem 3.2.13]{EJ1}, and the last one is from Remark 2.16. From this isomorphisms, it follows that $ \Hom_{R_{\fm}}\big( \Ext^i_{R_{\fm}}(k(\fm) , C_{\fm}) ,  E(k(\fm)) \big) = 0 $ for all $ i \neq \h(\fm) $, from which we conclude that $ \Ext^i_{R_{\fm}}(k(\fm) , C_{\fm}) = 0 $ for all $ i \neq \h(\fm) $, since $ E(k(\fm)) $ is an injective cogenerator in the category of $ R_{\fm} $-modules. Thus $ C_{\fm} $ is dualizing for $ R_{\fm} $, as required.
\end{proof}

\begin{cor}
	Let $C$ be pointwise dualizing. Then flat cover of any $C$-injective $R$-module is $C$-injective. 
\end{cor}
\begin{proof}
By Lemma 4.1,  $ C $-$ \id_R(T_{\fp}) = 0$ for any prime ideal $ \fp $ with $ \h(\fp) = 0 $. Hence $ T_{\fp} $ is $C$-injective for any prime ideal $ \fp $ with $ \h(\fp) = 0 $.	Assume that $M$ is a $C$-injective $R$-module. By Theorem 4.3, we have $ F(M) = \underset{\h(\fp) = 0} \prod T_{\fp} $. Now the result follows since the class $ \mathcal{I}_C $ closed under arbitrary direct product.
\end{proof}
\begin{cor}
The $R$-module $C$ is pointwise dualizing if and only if for any prime ideal $\fp$ of $R$,
  $$\pi_i \big(\fp, \emph\Hom_R(C,E(R/ \fp)) \big) = \left\lbrace
  \begin{array}{c l}
  1\ \ & \text{ \ \ \ \ \ $i=\emph\h(\fp)$,}\\
  0\ \ & \text{ \ \ \ \ \ $i\neq \emph\h(\fp).$}\\
  \end{array}
  \right.$$\\
\end{cor}
\begin{proof}
	Assume that $ \fp \in \Spec(R)$. Set $ k(\fp) = R_{\fp} / {\fp} R_{\fp} $. We have the following equalities
 \[\begin{array}{rl}
 \pi_i \big(\fp, \Hom_R(C,E(R/ \fp)) \big) &= \vdim_{k(\fp)} \Tor^{R_{\fp}}_i \big(k(\fp),\Hom_R(R_{\fp} , \Hom_R(C , E(R/ \fp)) \big) \\
 &=\vdim_{k(\fp)} \Tor^{R_{\fp}}_i \big(k(\fp) ,\Hom_R(C_{\fp} , \Hom_{R_{\fp}}(R_{\fp} , E(R/ \fp)) \big) \\
 &=\vdim_{k(\fp)} \Tor^{R_{\fp}}_i \big(k(\fp) ,\Hom_R(C_{\fp} , E(R/ \fp)) \big)\\
 &=\vdim_{k(\fp)} \Hom_{R_{\fp}}\big(\Ext_{R_{\fp}}^i (k(\fp) , C_{\fp}), E(R/ \fp) \big),\\
 \end{array}\]	
 where the second equality is from Remark 2.16, and the last equality is from \cite[Theorem 3.2.13]{EJ1}. Now, $C$ is pointwise dualizing if and only if $ C_{\fp} $ is the dualizing module of $ R_{\fp} $ for all $ \fp \in \Spec(R)$, and this is the case if and only if
 $$\Ext^i_{R_{\fp}} (k(\fp) ,C_{\fp}) \cong \left\lbrace
 \begin{array}{c l}
 k(\fp)\ \ & \text{ \ \ \ \ \ $i=\h(\fp)$,}\\
 0\ \ & \text{ \ \ \ \ \ $i\neq \h(\fp).$}\\
 \end{array}
 \right.$$\\
for all $ \fp \in \Spec(R)$. Thus we are done by the above equalities and the fact that $ \Hom_{R_{\fp}}(k(\fp) , E(R / \fp)) \cong k(\fp) $. 
\end{proof}
In the following corollaries, we are concerned with the local cohomology. For an $R$-module $M$, the $ i $-th local cohomology module of $M$ with respect to an ideal $ \fa $ of $R$, denoted by $ \H^i_{\fa}(M) $, is defined to be \\
\centerline{$  \H^i_{\fa}(M) = \underset{\underset{n \geq 1} \longrightarrow}  \lim \Ext^i_R(R/ \fa^n , M) $.}
For the basic properties of local cohomology modules, please see the textbook \cite{BS}.
\begin{cor}
Let $ (R, \fm) $ be a Cohen-Macaulay local ring with $ \emph{\dim}(R) = d $ possesing a canonical module $ \omega_R $. Then 
$ \pi_i \big(\fm, \emph\H_{\fm}^d(R) \big) = \delta_{i,d}$, and $ \pi_i \big(\fq, \emph\H_{\fm}^d(R) \big) = 0 $ for any non-maximal prime ideal $ \fq $ whenever $ i \neq \emph{\h}(\fq) $. 
\end{cor}
\begin{proof}
By \cite[Theorem 11.2.8]{BS}, we have $\H^d_{\fm}(R) \cong \Hom_R(\omega_R , E(R/ \fm))$, and hence $\H^d_{\fm}(R)$ is $\omega_R$-injective. Assume that $ \fq $ is a non-maximal prime ideal of $R$. Then by the Theorem 4.3, we have $ \pi_i \big(\fq, \H_{\fm}^d(R) \big) = 0 $ for all $ i \neq \h(\fq) $. Finally, by corollary 4.5, we have $ \pi_i \big(\fm, \H^d_{\fm}(R)\big) = 0 $ for all $ i \neq d $ and that $ \pi_d \big(\fm, \H^d_{\fm}(R)\big) = 1 $, as wanted.
\end{proof}
If  $ (R, \fm) $ is a Cohen-Macaulay local ring with $ \dim(R) = d $, then by \cite[Corollary 6.2.9]{BS} the only non-vanishing local cohomology of $R$ with respect to $ \fm $ is $ \H_{\fm}^d(R) $. Also, if $R$ admits a canonical module, then by \cite[Proposition 9.5.22]{EJ1}, we have $ \fd_R(\H_{\fm}^d(R)) = d  $. The following corollary describes the structure of the minimal flat resolution of $ \H_{\fm}^d(R) $.
\begin{cor}
	Let $ (R, \fm) $ be a $d$-dimensional Cohen-Macaulay local ring possessing a canonical module. The minimal flat resolution of $ \emph{\H}_{\fm}^d(R) $ is of the form \\
	\centerline{$ 0 \longrightarrow \widehat{R_{\fm}} \longrightarrow \cdots \longrightarrow \underset{\emph{\h}(\fp) = 1} \prod T_{\fp}  \longrightarrow \underset{\emph{\h}(\fp) = 0} \prod T_{\fp} \longrightarrow \emph{\H}_{\fm}^d(R) \longrightarrow 0 $.}
\end{cor}
In the following corollary, we give another proof of \cite[Corollary 3.7]{RT}. Our approach is direct, and uses the well-known fact that the homology functor $ \Tor $ can be computed by a flat resolution.  
\begin{cor}\label{P}
	Let $(R, \fm)$ be a $d$-dimensional Cohen-Macaulay local ring. Then
	$$\emph\Tor^R_i(\emph\H^d_{\fm}(R),\emph\H^d_{\fm}(R)) \cong \left\lbrace
	\begin{array}{c l}
	\emph\H^d_{\fm}(R)\ \ & \text{ \ \ \ \ \ $i=d$,}\\
	0\ \   & \text{  \ \ $\text{ \ \ $i \neq d$}$.}
	\end{array}
	\right.$$\\
\end{cor}
\begin{proof}
Note that $ \widehat{R} $ is a $ d $-dimensional complete Cohen-Macaulay local ring, and hence admits a canonical module $\omega _{\widehat{R}}$. The $R$-module $ \H^d_{\fm}(R) $ is Artinian by \cite[Theorem 7.1.6]{BS}, and thus naturally has a $ \widehat{R} $-module structure by \cite[Remark 10.2.9]{BS}. Hence $ \Tor^R_i(\H^d_{\fm}(R) , \H^d_{\fm}(R)) $ is Artinian for all $ i \geq 0 $ by \cite[Corollary 3.2]{KLW1}. Thus there are isomorphisms
\[\begin{array}{rl}
\Tor^R_i(\H^d_{\fm}(R),\H^d_{\fm}(R)) &\cong \Tor^R_i(\H^d_{\fm}(R),\H^d_{\fm}(R)) \otimes_R \widehat{R}  \\
&\cong \Tor^{\widehat{R}}_i(\H^d_{\fm}(R) \otimes_R \widehat{R} ,\H^d_{\fm}(R) \otimes_R \widehat{R})\\
&\cong \Tor^{\widehat{R}}_i \big(\H^d_{\fm\widehat{R}}(\widehat{R}) ,\H^d_{\fm\widehat{R}}(\widehat{R}) \big), \\
\end{array}\]
in which the second isomorphism is from \cite[Theorem 2.1.11]{EJ1}, and the last one is flat base change \cite[Theorem 4.3.2]{BS}. Also, we have the isomorphisms 
\[\begin{array}{rl}
\H^d_{\fm}(R) &\cong \H^d_{\fm}(R) \otimes_R \widehat{R} \\
&\cong \H^d_{\fm \widehat{R}}(\widehat{R})\\
&\cong \Hom_{\widehat{R}} \big(\omega _{\widehat{R}}, E_{\widehat{R}}(\widehat{R}/ \fm\widehat{R}) \big), \\
\end{array}\]
in which the first isomorphism holds because $ \H^d_{\fm}(R) $ is Artinian, the second isomorphism is the flat base change, and the last one is local duality \cite[Theorem 11.2.8]{BS}. Thus $\H^d_{\fm}(R)$ 
is a $\omega _{\widehat{R}}$-injective $ \widehat{R} $-module. Hence, by Corollary 4.7, the minimal flat resolution of $ \H_{\fm}^d(R) $, as an $ \widehat{R} $-module, is of the form \\
\centerline{$ 0 \longrightarrow \widehat{\widehat{R}_{\fm\widehat{R}}} \longrightarrow \cdots \longrightarrow \underset{\h(Q) = 1} \prod T_{Q}  \longrightarrow \underset{\h(Q) = 0} \prod T_{Q} \longrightarrow \H_{\fm}^d(R) \longrightarrow 0 $,}
in which $ T_{Q} $ is the completion of a free $ \widehat{R}_{Q} $-module with respect to $ Q\widehat{R}_{Q} $-adic topology, for $ Q \in \Spec(\widehat{R}) $. Observe that the above resolution is a flat resolution of  $\H^d_{\fm}(R)$ as an $R$-module since the modules in the above resolution are all flat $R$-modules. Therefore, we can replace $ R $ by $ \widehat{R} $, and assume that $ R $ is complete. So that, the minimal flat resolution of $ \H_{\fm}^d(R) $ is of the form \\
\centerline{$ 0 \longrightarrow \widehat{R_{\fm}} \longrightarrow \cdots \longrightarrow \underset{\h(\fp) = 1} \prod T_{\fp}  \longrightarrow \underset{\h(\fp) = 0} \prod T_{\fp} \longrightarrow \H_{\fm}^d(R) \longrightarrow 0 $,}
in which $ T_{\fp} $ is the completion of a free $ R_{\fp} $-module with respect to $ \fp R_{\fp} $-adic topology, for $ \fp \in \Spec(R) $.
Next, note that for each prime ideal $ \fp $ with $ \fp \neq \fm $, we have $ \H_{\fm}^d(R) \otimes_R \Big( \prod T_{\fp} \Big) = 0  $. Indeed, we can write
$ \H_{\fm}^d(R) = \underset{\underset{\alpha \in I} \longrightarrow}  \lim M_{\alpha}$, where $M_{\alpha}$ is a finitely generated submodule of $\H_{\fm}^d(R)$. Also $ T_{\fp} = \Hom_R\big(E(R/ \fp) , E(R/ \fp)^{(X)} \big) $ for some set $X$. Now since $ M_{\alpha} $ is of finite length by \cite[Theorem 7.1.3]{BS},  we can take an element
$ x \in \fm \smallsetminus \fq $ such that $ xM_{\alpha} = 0 $. But multiplication of $ x $ induces an automorphism on $ E(R/ \fp) $ and hence on $ \prod T_{\fp} $. Consequently, multiplication of $ x $ on $ M_{\alpha} \otimes_R \Big( \prod T_{\fp} \Big) $ is both an isomorphism and zero. Hence
$ M_{\alpha} \otimes_R \Big( \prod T_{\fp} \Big) = 0$, from which we conclude that $ \H_{\fm}^d(R) \otimes_R \Big( \prod T_{\fp} \Big) = 0  $ since tensor commutes with direct limit. Thus $ \Tor^R_i(\H^d_{\fm}(R),\H^d_{\fm}(R))  = 0 $ for $ i \neq d $. Finally, we have \\
  \[\begin{array}{rl}
  \Tor^R_d(\H^d_{\fm}(R),\H^d_{\fm}(R)) &\cong \widehat{R_{\fm}} \otimes_R \H^d_{\fm}(R) \\
  &\cong \H^d_{\fm \widehat{R_{\fm}}}(\widehat{R_{\fm}}) \\
  &\cong \Hom_{\widehat{R_{\fm}}}(\widehat{\omega_{R_{\fm}}} , E_{\widehat{R_{\fm}}}(\widehat{R_{\fm}}/ \fm \widehat{R_{\fm}}))\\
  	&\cong \Hom_{R_{\fm}}(\omega_{R_{\fm}} , E_{R_{\fm}}(R_{\fm}/ \fm R_{\fm})) \otimes_{R_{\fm}} \widehat{R_{\fm}} \\
  	&\cong \Hom_{R_{\fm}}(\omega_{R_{\fm}} , E_{R_{\fm}}(R_{\fm}/ \fm R_{\fm}))\\
  	&\cong \Hom_R(\omega_R , E(R/ \fm)) \otimes_R R_{\fm}\\
  	&\cong \Hom_R(\omega_R , E(R/ \fm) \otimes_R R_{\fm})\\
  	&\cong \Hom_R(\omega_R , E(R/ \fm))\\
  	&\cong \H^d_{\fm}(R),\\
  	\end{array}\]
  	in which the second isomorphism is the flat base change \cite[Theorem 4.3.2]{BS}, the third isomorphism is local duality \cite[Theorem 11.2.8]{BS}, and the fifth one is from \cite[Remark 10.2.9]{BS}, since $ \Hom_{R_{\fm}}(\omega_{R_{\fm}} , E_{R_{\fm}}(R_{\fm}/ \fm R_{\fm})) $ is an Artinian
  	$ R_{\fm} $-module and hence has a natural structure as an $ \widehat{R_{\fm}} $-module.
  \end{proof}
  The following theorem is a slight generalization of \cite[Theorem 3.3]{X2}. 
\begin{thm}
The following are equivalent:
\begin{itemize}
	\item[(i)]{$C$ is pointwise dualizing.}
	\item[(ii)] {If $M$ is a cotorsion $R$-module such that $C\emph{-\id}_R(M) = n < \infty$, then $M$ admits a minimal flat resolution such that $ \pi_i(\fp,M) = 0 $  for all $ \fp \in \emph{\Spec}(R) $ whenever $ \emph{\h}(\fp) \notin \{ i , ... , i + n \} $.}
\end{itemize}
\end{thm}
\begin{proof}
(i) $ \Longrightarrow $	(ii). We use induction on $n$. If $ n = 0 $, then we are done by Theorem 4.3. Now assume inductively that $ n > 0 $ and the case $ n $ ie settled. Fix a prime ideal $\fp$ of $R$. Assume that $ M $ is a cotorsion $R$-module with $ C $-$ \id_R(M) = n + 1 $. Hence $ M \in \mathcal{A}_C(R) $, and so the
$ \mathcal{I}_C $-preenvelope of $M$ is injective by \cite[Corollary 2.4(b)]{TW}. Thus there exists an exact sequence \\
\centerline{$ 0 \rightarrow M \rightarrow \Hom_R(C , I) \rightarrow L \rightarrow 0$, $ (*) $}
in which $I$ is injective, and $ L = \coker(M \rightarrow \Hom_R(C , I)) $. Note that $ L $ is cotorsion since both $M$ and $ \Hom_R(C , I) $ are cotorsion. Also, since both $ M$ and  $ \Hom_R(C , I) $ are in $\mathcal{A}_C(R) $, we have $ L \in \mathcal{A}_C(R) $, and therefore
$ \Tor^R_1(C,L) = 0 $. On the other hand $ C \otimes_R \Hom_R(C , I) \cong I$, by \cite[Theorem 3.2.11]{EJ1}. Hence application of $ C \otimes_R - $ on $ (*) $ yields an exact sequence \\
\centerline{$ 0 \rightarrow C \otimes_R M \rightarrow I \rightarrow C \otimes_R L \rightarrow 0$.}
By Theorem 2.9(i), we have $ \id_R(C \otimes_R M) = n + 1 $. Therefore $ \id_R(C \otimes_R L) = n $, whence $ C $-$ \id_R(L) = n $. Now induction hypothesis applied to $ \Hom_R(C , I) $ and $L$ yields that $ \pi_i \big(\fp, \Hom_R(C , I)\big) = 0  $ for all $ i \neq \h(\fp) $, and that $ \pi_i(\fp,L) = 0 $ fo all $ \h(\fp) \notin \{ i , ... , i + n \} $. Note that $ \Ext^1_R(R_{\fp} , M) = 0 $ since $M$ is cotorsion. Hence the exact sequence $ (*) $ yields an exact sequence \\
\centerline{$ 0 \rightarrow \Hom_R(R_{\fp} ,M) \rightarrow \Hom_R(R_{\fp} ,\Hom_R(C , I)) \rightarrow \Hom_R(R_{\fp} ,L) \rightarrow 0$,}
and the later exact sequence, by applying $ k(\fp) \otimes_{R_{\fp}} - $, yields the long exact sequence\\
\centerline{$ \cdots \rightarrow \Tor^{R_{\fp}}_{i+ 1} \big(k(\fp) , \Hom_R(R_{\fp} , \Hom_R(C , E)) \big) \rightarrow \Tor^{R_{\fp}}_{i+ 1} \big(k(\fp) , \Hom_R(R_{\fp} , L) \big) \rightarrow$}
\centerline{$\Tor^{R_{\fp}}_i \big(k(\fp) , \Hom_R(R_{\fp} , M) \big) \rightarrow \Tor^{R_{\fp}}_i \big(k(\fp) , \Hom_R(R_{\fp} , \Hom_R(C , E)) \big) \rightarrow \cdots $.}
From the above long exact sequence, it follows that $ \Tor^{R_{\fp}}_i \big(k(\fp) , \Hom_R(R_{\fp} , M) \big) = 0 $ for all $ \h(\fp) \notin \{ i , ... , i + n + 1 \} $, as wanted. This completes the inductive step. \\
 (ii) $ \Longrightarrow $ (i). Let $ \fm $ be a maximal ideal of $R$. Now $ \Hom_R(C,E(R/ \fm)) $ is $C$-injective and hence by assumption $ \pi_i \big(\fm, \Hom_R(C,E(R/ \fm))\big) = 0  $ for all $ i \neq \h(\fm) $. Now by the same argument as in the proof of Theorem 4.3, we have $ \Ext^i_{R_{\fm}}(k(\fm) , C_{\fm}) = 0 $ for all $ i \neq \h(\fm) $, whence $C_{\fm}$ is dualizing for $R_{\fm}$.
\end{proof}

\begin{cor}
The following statements hold true:
\begin{itemize}
	\item[(i)]{If $C$ is pointwise dualizing, then  $ C\emph{-\id}_R(F(M)) \leq C\emph{-\id}_R(M)$ for any cotorsion $R$-module $M$.}
	\item[(ii)] {If $ C\emph{-\id}_R(F(M)) \leq C\emph{-\id}_R(M)$ for any $R$-module $M$, then $C$ is pointwise dualizing.}
\end{itemize}	
\end{cor}
\begin{proof}
(i). Assume that $M$ is a cotorsion $R$-module. If $C$-$\id_R(M) = \infty$, then we are done. So assume that $C$-$\id_R(M) = n < \infty$. Then by Theorem 4.9, we have $ F(M) = \prod T_{\fp} $ where $ 0 \leq \h(\fp) \leq n $. Now the result follows by Lemma 4.1.

(ii). Assume that $ \fm $ is a maximal ideal of $R$. We have to show that $ C_{\fm} $ is dualizing for $ R_{\fm} $. Assume that \textbf{x} is a maximal $R$-sequence in $\fm$. Then $ \fd_R(R/ \textbf{x}R) < \infty $, and $ \Ass_R(C/ \textbf{x}C) = \{\fm \} $ since \textbf{x} is also a maximal $ C $-sequence. Hence we have the equalities
 \[\begin{array}{rl}
 C\emph{-}\fd (C /\textbf{x}C) &= \fd_R(\Hom_R(C, C /\textbf{x}C)) \\
 &=\fd_R(\Hom_R(C, C \otimes_R R /\textbf{x}R)) \\
 &=\fd_R(R /\textbf{x}R)\\
 &< \infty,\\
 \end{array}\]
 in which the first equality is from Theorem 2.9(ii), and the third one holds because  $ R/ \textbf{x}R \in \mathcal{A}_C(R) $. Assume that $E$ is an injective cogenerator. Set $ (-)^{\vee} = \Hom_R(-, E)$. Then $ C $-$ \id_R((C /\textbf{x}C)^{\vee}) < \infty $ by Lemma 2.10(ii). Now if $F$ is the flat cover of $ (C /\textbf{x}C)^{\vee} $, then by assumption, we have $ C $-$ \id_R(F) < \infty $. Therefore, we have $ C $-$ \fd_R(F^{\vee}) < \infty $ by Lemma 2.10(i).  Next, note that we have \\
 \centerline{$ C /\textbf{x}C \hookrightarrow (C /\textbf{x}C)^{\vee \vee} \hookrightarrow F^{\vee}$.}
 Hence, the injective envelope of $ C /\textbf{x}C $ is a direct summand of $ F^{\vee} $. Thus, in fact, $ E(R/ \fm) $ is a direct summand of $ F^{\vee} $, since $ R/ \fm \hookrightarrow C /\textbf{x}C $. It follows that $ C $-$ \fd_R(E(R/ \fm)) < \infty $, and hence we are done by Lemma 4.1, since $\fm$ was arbitrary.
\end{proof}

\textbf{Acknowledgments.} We thank the referee for very careful reading of the manuscript and also for his/her useful
suggestions.
 		
\bibliographystyle{amsplain}

\end{document}